\documentclass[preprint,authoryear,11pt]{elsarticle}
\usepackage{amsmath}
\usepackage{amssymb}

\newcommand{\tr}{\ast}

\newcommand{\aaa}{\mathbf{a}}
\newcommand{\xx}{\mathbf{x}}
\newcommand{\yy}{\mathbf{y}}
\newcommand{\bb}{\mathbf{b}}

\newcommand{\ee}{\mathbf{e}}
\newcommand{\vv}{\mathbf{v}}
\newcommand{\Span}{\mathrm{Span}}
\newcommand{\rank}{\mathrm{rank}}
\newcommand{\degree}{\mathrm{degree}}

\newcommand{\JF}{J_F}
\newcommand{\mux}{\mu}
\newcommand{\SJ}{\tilde{J}_{F_0}}
\newcommand{\SJF}{\tilde{J}_{F}}
\newtheorem{example}{EXAMPLE}[section]
\newtheorem{theorem}{Theorem}[section]

\newtheorem{lemma}[theorem]{Lemma}
\newtheorem{algorithm}[theorem]{Algorithm}
\newtheorem{remark}[theorem]{Remark}
\newtheorem{define}[theorem]{Definition}
\newtheorem{assumption}[theorem]{Assumption}
\newenvironment{proof}{{\bf{Proof. }}}{\hfill $\square$\\}

\journal{ }

\begin{document}

\begin{frontmatter}

\title{Verified Error Bounds for Isolated Singular Solutions of Polynomial Systems: Case of Breadth One\tnoteref{t1}}
\tnotetext[t1]{ This research is supported  by a NKBRPC 2011CB302400
and the Chinese National
 Natural Science Foundation under Grants 60821002/F02, 60911130369 and
 10871194.
Some results of this paper have been presented at the Symbolic and
Numerical Computation (SNC 2011) conference held June 7-9, 2011 in
San Jose, California.}

\author[klmm,inria]{Nan Li}
\ead{linan08@amss.ac.cn}
\author[klmm]{Lihong Zhi}
\ead{lzhi@mmrc.iss.ac.cn}
%\ead[url]{http://www.mmrc.iss.ac.cn/~lzhi/}

\address[klmm]{KLMM, Academy of Mathematics and System Science, CAS,
100190 Beijing, China}
\address[inria]{INRIA Paris-Rocquencourt, 78153 Le Chesnay Cedex, France}

\begin{abstract}
 In this
paper we describe how to improve the performance of the
symbolic-numeric method in \citep{LiZhi:2009,LZ:2011} for
computing the multiplicity structure and refining approximate
isolated singular solutions in the breadth one case. By introducing  a parameterized and
deflated system with smoothing parameters, we generalize the
algorithm in \citep{RuGr09}  to compute verified  error bounds
such that   a slightly perturbed polynomial system  is guaranteed
to have a breadth-one multiple root within the computed bounds.
\end{abstract}


\end{frontmatter}

\section{Introduction}
\label{intro}

It is  a challenge problem to solve the polynomial systems with
singular solutions.  Rall~\citep{Rall66}  studied  the convergence
properties of Newton's method at singular solutions, and a lot of
modifications of Newton's method to restore quadratic convergence
have  been proposed
in~\citep{DeckerKelley:1980I,DeckerKelley:1980II,DeckerKelley:1982,
Griewank:1980,  GriewankOsborne:1981, Griewank85, OWM:1983, Ojika:1987,
Reddien:1978,Reddien:1980, YAMAMOTONORIO:1984,
Chen97,ShenYpma05}. Recently,
 many new symbolic-numeric methods have been proposed for
refining  an approximate singular solution to high
accuracy~\citep{Corless:1997,DZ:2005,DLZ:2009,GLSY:2005,
GLSY:2007, Lecerf:2002, LVZ:2006, LVZ07, LVZ:2008, WuZhi:2008,
WuZhi:2009, MM:2011}. Especially, in  \citep{RuGr09,MM:2011}, they
computed verified error bounds for singular solutions of nonlinear
systems.

In \citep{LZ:2011},  we present a symbolic-numeric method to
refine an approximate isolated singular solution of a polynomial
system when the Jacobian matrix of the system evaluated at the
singular solution  has corank one  approximately. Our approach is
based on the regularized Newton iteration and the computation of
differential conditions satisfied at the approximate singular
solution. The size of matrices involved in our algorithm is
bounded by the number of variables. The algorithm will converge
quadratically if the approximate singular solution  is close to
the isolated exact singular solution.  A preliminary
implementation performs well in most cases.  However, it  may
suffer from computing and storing dense multiplicity structures
caused by linear transformation or dense expressions of
differential functionals even for sparse input polynomials. In
\citep{Linan:2011}, we show briefly how to evaluate the
differential conditions more efficiently by avoiding the linear
transformation and solving a sequence of least squares problems.
The techniques  for constructing a parameterized deflation system
and evaluations of differential conditions   are  similar to those
introduced in \citep{LVZ:2006, LVZ07}.

\paragraph{\bf{Main contribution}}
In this paper, we still  focus on the special case where the Jacobian
matrix has corank one.   We describe how to preserve the sparse
structure of the input polynomial systems by avoiding the linear
transformation. We reduce the storage space  for computing  the
multiplicity structure by saving and evaluating differential
conditions instead of explicit construction of differential
functionals. Furthermore, we show that the parameterized deflated
system introduced in \citep{Linan:2011} for avoiding the
construction of the differential functionals repeatedly can be
used to  generalize the algorithm in \citep{RuGr09}  to compute
verified  error bounds, therefore,   a slightly perturbed
polynomial system  is guaranteed to have a breadth-one multiple
root within the computed bounds. We prove that it is always
possible to construct a regular augmented system to compute an
inclusion of the singular root by choosing properly smooth
parameters and renumbering the polynomials. We provide numerical
experiments to demonstrate the effectiveness of our method.

\paragraph{\bf{Structure of the paper}}
Section \ref{pre} is devoted to recall some notations and
well-known facts. In Section \ref{dual}, we describe a new
algorithm for computing the multiplicity structure of the singular
solution when the Jacobian matrix has corank one. Some experiment
results are given to show the efficiency of the new algorithm. In
Section \ref{verify}, we show how to construct a parameterized
deflated system  to refine and  compute verified error bounds for
the breath-one multiple roots. Some numerical examples are given
to demonstrate the performance of our algorithm. %We mention some
%ongoing research in Section \ref{con}.

\section{Notation and Preliminaries}
\label{pre}

Let $R=\mathbb{K}[\xx]$ denote a polynomial ring over the field
$\mathbb{K}$ of characteristic zero.  Let $I=(f_1,\ldots,f_n)$ be
an ideal of $R$, $\hat{\xx}\in\mathbb{K}^n$ an isolated root of
$I$, $m_{\hat{\xx}}=(x_1-\hat{x}_1,\ldots,x_n-\hat{x}_n)$ the
maximal ideal at $\hat{\xx}$. Suppose $Q_{\hat{\xx}}$ is the
isolated primary component whose associate prime is
$m_{\hat{\xx}}$, then the multiplicity $\mux$ of $\hat{\xx}$ is
defined as the dimension of the quotient ring $R/Q_{\hat{\xx}}$.

Let  $\mathbf{d}^{\alpha}_{\hat{\xx}}: R\rightarrow \mathbb{K}$ denote
the differential functional defined by
\begin{equation}
\mathbf{d}^{\alpha}_{\hat{\xx}}(g)=\frac{1} {\alpha_1!\cdots
\alpha_n!}\cdot\frac{\partial^{|\alpha|} g}{\partial
x_1^{\alpha_1}\cdots
\partial x_n^{\alpha_n}}(\hat{\xx}),\quad\forall g(\xx)\in R,
\end{equation}
for a point $\hat{\xx}\in \mathbb{K}^n$ and an array $\alpha\in
\mathbb{N}^n$. The normalized differentials have a useful
property: when $\hat{\xx}=\mathbf{0}$, we have
$\mathbf{d}^{\alpha}_{\mathbf{0}}(\xx^{\beta})=1$ if
$\alpha=\beta$ or $0$ otherwise. We may occasionally write
$\mathbf{d}^{\alpha}=d_1^{\alpha_1}d_2^{\alpha_2}\cdots
d_n^{\alpha_n}$ instead of $\mathbf{d}^{\alpha}_{\hat{\xx}}$ for
simplicity if $\hat{\xx}$ is clear from the context, where
$d_i^{\alpha_i}=\frac{1}{{\alpha_i}!}\frac{\partial^{\alpha_i}}{\partial
x_i^{\alpha_i}}$.

%
%We can now define the local dual space of an isolated point.

\begin{define}
The local dual space of $I$ at $\hat{\xx}$ is the subspace of elements of
$\mathfrak{D}_{\hat{\xx}}=\Span_\mathbb{K}\{\mathbf{d}^{\alpha}_{\hat{\xx}},\alpha\in\mathbb{N}^n\}$ that vanish on all the elements of $I$
\begin{equation}
\mathcal{D}_{\hat{\xx}}:=\{\Lambda\in \mathfrak{D}_{\hat{\xx}}\,\,|\,\,
\Lambda(f)=0, ~\forall f\in I\},
\end{equation}
where $\dim(\mathcal{D}_{\hat{\xx}})=\mux$.
\end{define}

Computing a closed basis of the local dual space is done essentially
by matrix-kernel computations
\citep{MMM:1995,Mourrain:1996,DZ:2005,WuZhi:2008,Zeng:2009}, which
are based on the \emph{stability} \emph{property} of
$\mathcal{D}_{\hat{\xx}}$:
\begin{equation}\label{closed}
\forall
\Lambda\in\mathcal{D}_{\hat{\xx}}^t,\,\,\Phi_{x_i}(\Lambda)\in\mathcal{D}_{\hat{\xx}}^{t-1},\,\,
i=1,\ldots,n,
\end{equation}
where $\mathcal{D}_{\hat{\xx}}^t$ denotes the subspace of $\mathcal{D}_{\hat{\xx}}$ of the degree less than or equal to $t$, for $t\in \mathbb{N}$, and $\Phi_{x_i}:\mathfrak{D}_{\hat{\xx}}\rightarrow\mathfrak{D}_{\hat{\xx}}$ are the linear \emph{anti-differentiation} \emph{operators} defined by
\begin{align*}
\Phi_{x_i}(\mathbf{d}^{\alpha}_{\hat{\xx}}):=\left\{
\begin{array}{ll}
\mathbf{d}^{(\alpha_1, \ldots, \alpha_{i}-1, \ldots, \alpha_n)}_{\hat{\xx}}, & \mbox{if $\alpha_{i} > 0$,}\\
0, & \mbox{otherwise.}
\end{array}
\right.
\end{align*}
\begin{lemma}\citep[Theorem 8.36]{Stetter:2004}\label{coeff} Suppose $\{\Lambda_1,\ldots,\Lambda_s\}$ is a closed basis of $\mathcal{D}_{\hat{\xx}}^{t-1}$, then an element $\Lambda\in\mathfrak{D}_{\hat{\xx}}$ lies in $\mathcal{D}_{\hat{\xx}}^{t}$ if and only if it satisfies (\ref{closed}) and $\Lambda(f_i)=0$ for $i=1,\ldots,n$.
\end{lemma}
In fact, (\ref{closed}) is equivalent to finding $\lambda_{i,k}\in\mathbb{K}$ such that $\Lambda\in\mathfrak{D}_{\hat{\xx}}$ satisfying
\begin{equation}\label{findcoeff}
\Phi_{x_i}(\Lambda)=\lambda_{i,1}\Lambda_1+\lambda_{i,2}\Lambda_2+\cdots+\lambda_{i,s}\Lambda_s,\mbox{
for }i=1,\ldots,n.
\end{equation}
If $\lambda_{i,k}$ are known, we can compute $\Lambda$ by the
following formula  \citep{Mourrain:1996}
\begin{equation}\label{constructL}
\Lambda=\sum_{j=1}^{s}\lambda_{1,j}\Psi_{x_1}(\Lambda_j)+\sum_{j=1}^{s}\lambda_{2,j}\Psi_{x_2}(\Lambda_j)+\cdots+\sum_{j=1}^{s}
\lambda_{n,j}\Psi_{x_n}(\Lambda_j),
\end{equation}
where the \emph{differentiation} \emph{operators} $\Psi_{x_i}:\mathfrak{D}_{\hat{\xx}}\rightarrow\mathfrak{D}_{\hat{\xx}}$ are defined by
\begin{align*}
\Psi_{x_i}(\mathbf{d}^{\alpha}_{\hat{\xx}}):=\left\{
\begin{array}{ll}
\mathbf{d}^{(\alpha_1, \ldots, \alpha_{i}+1, \ldots, \alpha_n)}_{\hat{\xx}}, & \mbox{if $\alpha_{1}=\cdots=\alpha_{i-1}=0$,}\\
0, & \mbox{otherwise.}
\end{array}
\right.
\end{align*}

Here and hereafter, let  $J_F(\hat{\xx})$ denote the Jacobian
matrix of the polynomial system  $F$ evaluated at $\hat{\xx}$. It
has been noticed in \citep{Stanley:1973, DZ:2005} that when the
corank of the Jacobian matrix $\JF(\hat{\xx})$ is one,
$\mathcal{D}_{\hat{\xx}}$ has the important property:
\[\dim(\mathcal{D}_{\hat{\xx}}^{t})-\dim(\mathcal{D}_{\hat{\xx}}^{t-1})=1,
~\text{for}~ 1 \leq t \leq \mu-1.\] %This can be easily proved from properties
%$(i)$ and $(ii)$ of Lemma \ref{HF}, and obviously $N=\mux-1$ for
%this special case.
 Hence, it is also called the breadth one case in \citep{DZ:2005}.
 For this special case, in
\citep{LiZhi:2009}, under the assumption that the first column of
$\JF(\hat{\xx})$ is zero, we employ both normalization and reduction
techniques to compute a closed basis of $\mathcal{D}_{\hat{\xx}}$
very efficiently by solving $\mux-1$ linear systems with the size
bounded by $n\times (n-1)$.

\begin{theorem}\citep[Theorem 3.1]{LiZhi:2009}\label{withtran}
Suppose $\hat{\xx}$ is an isolated breath-one singular root of a
given polynomial system $F=\{f_1,\ldots,f_n\}$ with the
multiplicity $\mux$, the first column of the Jacobian matrix
 $\JF(\hat{\xx})$ is zero. Set $\Lambda_1=1$ and $\Lambda_2=d_1$, then we can construct $\Lambda_k$ incrementally for $k$ from $3$
by
\begin{equation}\label{L}
\Lambda_k=\Delta_k+a_{k,2}d_2+a_{k,3}d_3+\cdots+a_{k,n}d_n,
\end{equation}
where $\Delta_k$ is a differential functional which has no free parameters and can be
obtained from previous computed $\{\Lambda_1,\Lambda_2,\ldots,\Lambda_{k-1}\}$ by
\begin{equation}\label{constructP}
\Delta_k=\Psi_{x_1}(\Lambda_{k-1})+\sum_{j=2}^{k-1}a_{j,2}\Psi_{x_2}(\Lambda_j)+\cdots+\sum_{j=2}^{k-1}a_{j,n}\Psi_{x_n}(\Lambda_j).
\end{equation}
The parameters $a_{k,j}$, for $j=2,\ldots,n$, are determined by
solving
\begin{equation}\label{kthcoeff}
\SJF(\hat{\xx})\cdot\left[\begin{array}{c}
  a_{k,2}  \\
  \vdots   \\
  a_{k,n}
\end{array}
\right]=-\left[\begin{array}{c}
  \Delta_k(f_1)  \\
  \vdots   \\
  \Delta_k(f_n)
\end{array}
\right],
\end{equation}
where $\SJF(\hat{\xx})$ consists of the last $n-1$ columns of $\JF(\hat{\xx})$. This process will be
stopped if there is no
solution for (\ref{kthcoeff}).  We get the multiplicity $\mu=k-1$ and  %. At this moment, $k=\mux+1$
 $\{\Lambda_1,\Lambda_2,\ldots,\Lambda_{\mux}\}$ a closed basis of
the local dual space $\mathcal{D}_{\hat{\xx}}$.
\end{theorem}

In \citep{LiZhi:2009}, when the first column of $\JF(\hat{\xx})$ is
not zero, we apply a linear transformation of variables to obtain a
new system and a new root, which will satisfy the assumptions of
Theorem \ref{withtran}. Finally, we can derive a closed basis of the
local dual space of the original system at the original root by
transforming back the computed basis.  Unfortunately, these
transformations always result in  dense systems even the original
ones are sparse. % and the computed basis are always not well reduced.

\begin{example}\label{ex1}\citep{Ojika:1987}
Consider a polynomial system
\begin{equation*}
F=\left\{x_1^2+x_2-3,x_1+\frac{1}{8}x_2^2-\frac{3}{2}.\right\}
\end{equation*}
The system F has $(1,2)$ as a $3$-fold isolated zero.
\end{example}
The Jacobian matrix of $F$ at $(1,2)$ is
\[\JF(1,2)=\left[\begin{array}{cc}
  2 & 1  \\
  1 & \frac{1}{2}
\end{array}
\right],\] which has a non-trivial null vector
$\mathbf{r}=(-\frac{1}{2},1)^T$. We apply a linear
transformation of the variables
\[x_1=-\frac{1}{2}y_1+2y_2,x_2=y_1+y_2,\]
to obtain a new  dense  polynomial system
\[G=\left\{\frac{1}{4}y_1^2-2y_1y_2+4y_2^2+y_1+y_2-3,\frac{1}{8}y_1^2+\frac{1}{4}y_1y_2+\frac{1}{8}y_2^2-\frac{1}{2}y_1+2y_2-\frac{3}{2}.\right\}\]
The returned closed basis of the local dual space of $G$ at the new
point $(\frac{6}{5},\frac{4}{5})^T$ by Theorem \ref{withtran} is
\[\Lambda_1=1,\Lambda_2=d_1,\Lambda_3=d_1^2-\frac{1}{20}d_2,\]
which can be transformed back to a closed basis of $F$ at $(1,2)$
\[\Lambda_1=1,\Lambda_2=-\frac{1}{2}d_1+d_2,\Lambda_3=\frac{1}{4}d_1^2-\frac{1}{2}d_1d_2+d_2^2-\frac{1}{10}d_1-\frac{1}{20}d_2.\]

\section{A Modified Algorithm for Computing a Closed basis of the Local Dual Space}
\label{dual}

In this section, we  show how to avoid the linear transformations in
computing a closed basis
$\{\Lambda_1,\Lambda_2,\ldots,\Lambda_{\mux}\}$ of
the local dual space $\mathcal{D}_{\hat{\xx}}$. %We select a column
%of $\JF(\hat{\xx})$, which can be written as a linear combination
%of the other $n-1$ linearly independent columns, then use the
%respect variable to apply the normalization and reduction to avoid
%those linear transformations.

Let $\mathbf{r}=(r_1,r_2,\ldots,r_n)^T$ be a non-trivial null
vector of $\JF(\hat{\xx})$. Without loss of
generality, we assume
\begin{equation}\label{r1} |r_{1}|\geq |r_{j}|,\mbox{ for }1\leq j\leq n.
\end{equation}
Otherwise, one can perform changes of variables to guarantee
(\ref{r1}) is satisfied. Then we normalize $\mathbf{r}$ by $r_1$
and derive that
\begin{equation}\label{a2}
\mathbf{a}_2=\left[1,a_{2,2},\ldots,a_{2,n}\right]^T=\left[1,\frac{r_2}{r_1},\ldots,\frac{r_n}{r_1}\right]^T
\end{equation}
is also a non-trivial null vector of $\JF(\hat{\xx})$. We set
$\Lambda_1=1$ and
\begin{equation}\label{initialization}
\Lambda_2=d_1+a_{2,2}d_2+\cdots+a_{2,n}d_n, ~|a_{2,2}| \leq 1,
\ldots, |a_{2,n}| \leq 1.
\end{equation}
\begin{lemma}\label{reduction}
Under the assumption of (\ref{initialization}), the differential
functional monomial $d_1^{k-1}$ does not vanish in $\Lambda_k$,
for $k=3,\ldots,\mux$.
\end{lemma}
\begin{proof}
If $\mux\geq3$, $\Lambda_3$ will satisfy (\ref{closed}) and
(\ref{findcoeff}), and at least one of
$\lambda_{1,2}$,$\lambda_{2,2}$,$\ldots$, $\lambda_{n,2}$ is not zero.
Let $t$ be the integer such that  $\lambda_{t,2}\neq0$, then the
differential functional monomial $d_1 d_t$ does not vanish in
$\Lambda_3$ according to (\ref{constructL}). It follows that $d_t$
does not vanish in $\Phi_{x_1}(\Lambda_3)$.  Then from
(\ref{findcoeff}) and  (\ref{initialization}), we derive that
$\lambda_{1,2}\neq0$. Therefore, $d_1^2$ does not vanish in
$\Lambda_3$. % moreover, $a_{2,t}\neq0$.

The rest proof is done by induction. Assume the lemma is true for
$k$ and $k<\mux$, then similar to the analysis above,
$\Lambda_{k+1}$  satisfies   (\ref{closed}) and (\ref{findcoeff}).
Let $t$ be the integer such that $\lambda_{t,k}\neq0$, then
$d_1^{k-1}d_t$ does not vanish in $\Lambda_{k+1}$. It follows that
$d_1^{k-2}d_t$ does not vanish in $\Phi_{x_1}(\Lambda_{k+1})$.
Since
$\dim(\mathcal{D}_{\hat{\xx}}^{k})-\dim(\mathcal{D}_{\hat{\xx}}^{k-1})=1$,
 $\degree(\Lambda_k)=k-1$,  we derive that $\lambda_{1,k}\neq0$.
Therefore, $d_1^{k}$ does not vanish in $\Lambda_{k+1}$.
\end{proof}

\begin{remark}\label{normalized1}
According to  Lemma \ref{reduction}, the coefficient of $d_1^{k-1}$
in $\Lambda_k$ is not zero, then can be normalized  to be $1$.
Moreover, we can assume that $\Lambda_k$ does not have terms
$\{1,d_1,d_1^2,\ldots,d_1^{k-2}\}$ inside. Otherwise, one can reduce
them by $\{\Lambda_1,\Lambda_2,\ldots,\Lambda_{k-1}\}$. These
normalization and reduction can help us reduce the number of free
parameters in (\ref{findcoeff}) to $n-1$.
\end{remark}

\begin{lemma}\label{freepara}
Under the assumption of (\ref{initialization}) and after performing the normalization and reduction above, we have
\begin{equation} \label{phix1}
\left\{
\begin{array}{l}
\Phi_{x_1}(\Lambda_k)=\Lambda_{k-1}, \\
\Phi_{x_i}(\Lambda_k)=a_{k,i}\Lambda_1+a_{k-1,i}\Lambda_2+\cdots+a_{2,i}\Lambda_{k-1},\mbox{
for }i=2,\ldots,n,
\end{array}
\right.
\end{equation}
where $a_{j,i}$ is the coefficient of $d_i$ in $\Lambda_j$, for
$2\leq j\leq k$ and $k\leq \mux$.
\end{lemma}

\begin{proof}
By Lemma \ref{reduction} and Remark \ref{normalized1}, we know that
$\Lambda_k$ has  a term $d_1^{k-1}$ and there are  no terms of
$\{1,d_1,d_1^2,\ldots,d_1^{k-2}\}$ in $\Lambda_k$. Hence, according
to (\ref{findcoeff}), we derive that
$\Phi_{x_1}(\Lambda_k)=\Lambda_{k-1}$. Furthermore, since
$\Phi_{x_i}(\Lambda_k) \in \mathcal{D}_{\hat{\xx}}^{k-2}$, we have
\[\Phi_{x_i}(\Lambda_k)=\lambda_{i,1}\Lambda_1+\lambda_{i,2}\Lambda_2+\cdots+\lambda_{i,k-1}\Lambda_{k-1},
~\mbox{for}~ i=2,\ldots,n.\]
 Using (\ref{constructL}),
  we claim that $\lambda_{i,j}$  is equal to the coefficient of
% $=$\emph{coeff}
 $(d_1^{j-1}d_i)$ in $\Lambda_k$. On the other hand, we know that $\Phi_{x_1}^{j-1}(\Lambda_k)=\Lambda_{k-j+1}$, hence
  the coefficient of  %\emph{coeff}
 $(d_1^{j-1}d_i)$ in $\Lambda_k$ is equal to the coefficient of % $=$\emph{coeff}
 $d_i$ in $\Lambda_{k-j+1}$ which is equal to $a_{k-j+1,i}$. Hence,
 $\lambda_{i,j}=a_{k-j+1,i}$ for $2 \leq j \leq k-1$ and we prove
 the second equality in (\ref{phix1}).
\end{proof}

According to Lemma \ref{freepara}, from a closed basis
$\{\Lambda_1,\Lambda_2,\ldots,\Lambda_{k-1}\}$ of
$\mathcal{D}_{\hat{\xx}}^{k-2}$ to compute a new element
$\Lambda_{k}$ in
$\mathcal{D}_{\hat{\xx}}^{k-1}/\mathcal{D}_{\hat{\xx}}^{k-2}$,
 the only $n-1$ free parameters are $a_{k,i}$.
 Now we modify Theorem \ref{withtran} under the assumption (\ref{initialization}) to avoid the linear transformations.
 Note that we  also adopt a new equivalent form (\ref{newconstructP}) to construct $\Delta_{k}$ instead of (\ref{constructP}),
 which introduces an efficient method for  evaluating   the differential
 functionals.

\begin{theorem}\label{withouttran}\citep{Linan:2011}
Suppose $\hat{\xx}$ is an isolated breath-one singular root of a
given polynomial system $F=\{f_1,\ldots,f_n\}$ with the
multiplicity $\mux$. Set $\Lambda_1=1$ and
$\Lambda_2=d_1+a_{2,2}d_2+\cdots+a_{2,n}d_n$, then we can
construct $\Lambda_k$ incrementally for $k$ from $3$ by
\begin{equation}\label{newL}
\Lambda_k=\Delta_k+a_{k,2}d_2+a_{k,3}d_3+\cdots+a_{k,n}d_n,
\end{equation}
where
%$\Delta_k$ is a differential functional which has no free parameters and can be
%obtained from previous computed $\{\Lambda_1,\Lambda_2,\ldots,\Lambda_{k-1}\}$ by
\begin{equation}\label{newconstructP}
\Delta_k=\frac{1}{k-1}\left[\frac{\partial}{\partial x_1}\Lambda_{k-1}+\sum_{j=2}^n
\frac{\partial}{\partial
x_j}(a_{2,j}\Lambda_{k-1}+\cdots+(k-2)a_{k-1,j}\Lambda_{2})\right]
\end{equation}
The parameters $a_{k,j}$, for $j=2,\ldots,n$ are determined by solving (\ref{kthcoeff}),
%\begin{equation}\label{stillkthcoeff}
%\mbox{Jacobian}(\hat{\xx})_{:,2..n}\cdot\left[\begin{array}{c}
%  a_{k,2}  \\
%  \vdots   \\
%  a_{k,n}
%\end{array}
%\right]=-\left[\begin{array}{c}
%  \Delta_k(f_1)  \\
%  \vdots   \\
%  \Delta_k(f_n)
%\end{array}
%\right]
%\end{equation}
When $k=\mu+1$, there is no solution for (\ref{kthcoeff}) and the
process will be stopped.  The set
$\{\Lambda_1,\Lambda_2,\ldots,\Lambda_{\mux}\}$ is a closed basis of
the local dual space $\mathcal{D}_{\hat{\xx}}$.
\end{theorem}
\begin{proof}
According to Lemma \ref{freepara} and
(\ref{constructL},\ref{newL},\ref{newconstructP}), the constructed
$\Lambda_k$ satisfy the \emph{stability property} (\ref{closed}).
Moreover, solving (\ref{kthcoeff}) will guarantee that
$\Lambda_k(f_i)=0$, for $i=1,\ldots,n$. Therefore, by Lemma
\ref{coeff}, the set $\{\Lambda_1,\Lambda_2,\ldots,\Lambda_{\mux}\}$
is a closed basis of the local dual space $\mathcal{D}_{\hat{\xx}}$.
\end{proof}
\begin{remark}
As showed in \citep{LiZhi:2009}, for solving (\ref{kthcoeff}),
only the vector on the right side is updated, while the matrix of
the size $n\times(n-1)$ on the left side is fixed. So we  apply
the LU decomposition to $\SJF(\hat{\xx})$, then solve two
triangular systems instead of (\ref{kthcoeff}).
\end{remark}

Now we consider Example \ref{ex1} again. Since $\JF(1,2)$ has a
non-trivial null vector $\mathbf{r}=(-\frac{1}{2},1)^T$, we perform
a change of variables $x_1\leftrightarrow x_2$, then apply the
method described in Theorem \ref{withouttran} for computing a closed
basis of the local dual space of $F$ at $(1,2)$. We derive that
\[\Lambda_1=1,\Lambda_2=-\frac{1}{2}d_1+d_2,\Lambda_3=\frac{1}{4}d_1^2-\frac{1}{2}d_1d_2+d_2^2-\frac{1}{8}d_1.\]

In \citep{LiZhi:2009}, in order to compute $\Lambda_{k}$ of
$\mathcal{D}_{\hat{\xx}}^{k-1}/\mathcal{D}_{\hat{\xx}}^{k-2}$, we need to
construct $\Delta_{k}$ by (\ref{constructP}) and evaluate
$\Delta_{k}(f_i)$, for $i=1,\ldots,n$. Even if the input system
$F$ is sparse, the differential functional $\Delta_{k}$  could
still be very dense. Hence, the evaluation of the vector on the
right side of (\ref{kthcoeff}) could be very expensive sometimes.
\begin{example}\label{ex2}
Consider a polynomial  system $F=\{f_1,\ldots,f_s\}$
\begin{eqnarray*}
f_i&=&x_i^3+x_i^2-x_{i+1}, ~\mbox{if}~i<s,\\
f_s&=&x_s^2,
\end{eqnarray*}
with a breath-one singular zero $(0,\ldots,0)$ of the multiplicity
$2^s$.
\end{example}

As shown in \citep{LiZhi:2009}, for $s=6$, about $17$MB of memory
is used to store the local dual bases and it takes about $3$ hours
to compute all of them.  Moreover, for $s=7$, we are not able to
obtain all $\Lambda_k$ in $2$ days, and for $s=9$, the estimated
store space is about $1$GB. It is not a surprise that the
computation is dominated by the evaluation of $\Delta_{k}(F)$ in
(\ref{kthcoeff}).

In fact, we can view $\Delta_{k}$ and $\Lambda_k$ as differentiation
operators, denoted by $P_k$ and $L_k$ respectively
\citep{LiZhi:2009}, then we can take advantage of (\ref{newL}) and
(\ref{newconstructP}) to construct the polynomial systems $P_{k}(F)$
and $L_k(F)$ by
\begin{equation}\label{LkPk}
P_{k}(F)=\sum_{j=1}^{k-2}\frac{j}{k-1}\cdot
J_{L_{k-j}(F)}\cdot\mathbf{a}_{j+1}\mbox{ and
}L_{k}(F)=P_{k}(F)+J_{F}\cdot\mathbf{a}_k,
\end{equation}
where $J_{L_{j}(F)}$ is the Jacobian matrix of $L_{j}(F)$,
$\mathbf{a}_2=[1,a_{2,2},\ldots,a_{2,n}]^T$ and
$\mathbf{a}_j=[0,a_{j,2},\ldots,a_{j,n}]^T$ for $j=3,\ldots,k-1$.
Hence, we can compute the evaluation of $P_k(F)$ which is equal to
$\Delta_k(F)$ without constructing and storing the dense
differential functionals   $\Delta_{k}$.

The routine \verb"MSB1"  below takes an ideal
$I=(f_1,f_2,\ldots,f_n)\subset R$ and an isolated root
$\hat{\xx}\in\mathbb{K}^n$ of $I$ as input, where the corank of the
Jacobian matrix $\JF(\hat{\xx})$ is one, and returns the
multiplicity $\mux$ and a closed basis of the local dual space
$\mathcal{D}_{\hat{\xx}}$. Besides, we take
$\mathbf{a}_2,\mathbf{a}_3,\ldots,\mathbf{a}_{\mux}$ as output too,
since one can construct all $\Lambda_k$ by (\ref{newL}) and
(\ref{newconstructP}) immediately after they are computed. Another
reason for outputting $\mathbf{a}_i$'s is that these values are
important for multiple root refinement and verification if
$\hat{\xx}$ is given with limited precision, which will be discussed
in the next section.

\begin{algorithm}\label{MSB1}MSB1

\noindent \textbf{Input:} A polynomial system
$F=\{f_1,f_2,\ldots,f_n\}$ and a root $\hat{\xx}\in\mathbb{K}^n$.

\noindent \textbf{Output:} The multiplicity $\mu$, the parameters
$\mathbf{a}_2,\mathbf{a}_3,\ldots,\mathbf{a}_{\mux}$  and a closed
basis $\{\Lambda_1,\Lambda_2,\ldots,\Lambda_{\mux}\}$ of the local
dual space  $\mathcal{D}_{\hat{\xx}}$.

\begin{enumerate}

\item \label{ss1} Compute a non-trivial null vector
$\mathbf{r}=[r_1,r_2,\ldots,r_n]^T$ of $J_F(\hat{\xx})$.  Let $t$
be the integer s.t. $|r_{t}|\geq |r_{j}|$, $1\leq j\leq n$.

 Apply
variables exchange $x_1\leftrightarrow x_t$ to $F$, $\hat{\xx}$,
$J_F$ and $\mathbf{r}$. Set
    \[\mathbf{a}_2:=\left[1,\frac{r_2}{r_1},\frac{r_3}{r_1},\ldots,\frac{r_n}{r_1}\right]^T,
    L_2(F):=J_F\cdot\mathbf{a}_2\mbox{ and }P_3:=\frac{1}{2}J_{L_2(F)}\cdot\mathbf{a}_2.\]
    Compute the LU Decomposition of $\SJF(\hat{\xx})=P\cdot L\cdot U$. Set $k:=3$.

\item \label{ss2} Compute $P_k(F)$ by (\ref{LkPk}) and evaluate it
at $\hat \xx$ to get $\Delta_k(F)$, and
solve \[L\cdot\mathbf{b}=-P^{-1}\cdot\Delta_k(F). \]% ~{\text{where}}~ \mathbf{b}=(b_1,b_2,\ldots,b_n)^T.\]
 If the last entry in $\bb$ is zero, solve  $U_{1..(n-1),:}\cdot\mathbf{c}=\mathbf{b}$,
%where $\mathbf{c}=(c_1,c_2,\ldots,c_{n-1})^T$,
 and set
    \[\mathbf{a}_k:=\left[\begin{array}{l}
    0\\
    \mathbf{c}
    \end{array}\right]
    %c_1,c_2,\ldots,c_{n-1})^T
    ~\mbox{ and } L_k(F):=P_k(F)+J_F\cdot\mathbf{a}_k,\]
    and repeat with  $k:=k+1$. Otherwise, set  $\mux:=k-1$, go to Step \ref{ss4}.

%\item \label{s3} Set  and
%    \[\Delta_{k}(F):=\frac{1}{k-1}\left(\Lambda_{k-1}(F),2\Lambda_{k-2}(F),\ldots,(k-2)\Lambda_2(F)\right)
%    \left[\begin{array}{c}\mathbf{a}_2^T\\ \mathbf{a}_3^T\\ \vdots\\ \mathbf{a}_{k-1}^T\end{array}\right]
%    \left(\begin{array}{c}\frac{\partial}{\partial x_1}\\ \frac{\partial}{\partial x_2}\\ \vdots
%    \\ \frac{\partial}{\partial x_n}\end{array}\right).\]
%     Go back to Step \ref{s2}.

\item \label{ss4} Construct
$\{\Lambda_1,\Lambda_2,\ldots,\Lambda_{\mu}\}$ by (\ref{newL}) and
(\ref{newconstructP}) using   $\mathbf{a}_2$,$\mathbf{a}_3$,
$\ldots,\mathbf{a}_{\mux}$ incrementally, %starting from
%    \[\Lambda_1=1\mbox{ and }\Lambda_2=d_1+a_{2,2}d_2+\cdots+a_{2,n}d_n,\]
    then perform variables exchange $d_1\leftrightarrow d_t$.

%\item Return $\mux:=k-1$ and
%$\{\Lambda_1,\Lambda_2,\ldots,\Lambda_{\mux}\}$.

\end{enumerate}

\end{algorithm}

It should be noticed that  we run  Step \ref{ss4} in Algorithm
\ref{MSB1} only  when  a closed basis for the local dual space is
wanted. In general,  we can omit this step in refining and
certifying approximate singular solutions.  We have implemented this
method in Maple. In the following table, we show the time needed for
computing all $\mathbf{a}_j$ in Example \ref{ex2}, for
$j=2,\ldots,2^s$.
\begin{table}[ht]
\begin{center}
\begin{tabular}{|c|c|c|c|c|c|} \hline
s & 6 & 7 & 8 & 9 & 10\\ \hline
multiplicity & 64 & 128 & 256 & 512 & 1024\\ \hline
time(sec.) & 0.593 & 1.377 & 3.445 & 10.913 & 44.659\\ \hline\end{tabular}
%{\label{table1}Algorithm Performance}
\end{center}
\vspace{-0.4cm}
\end{table}

\section{Verified Multiple Roots of Polynomial Systems}
\label{verify} As mentioned in \citep{MM:2011}, in real-life
applications it is common to work with approximate inputs, and
usually we need to  decide numerically whether an approximate system
possesses a unique real or complex root in a given domain.

Standard verification methods for nonlinear systems are based on
the following theorem \citep{Krawczyk:1969, Moore:1977,
Rump:1983}.

\begin{theorem}\label{verification}
Let $F\in\mathbb{R}^n$ be a polynomial system with
$F=(f_1,\ldots,f_n)$, and $\tilde{\xx}\in\mathbb{R}^n$ a real
point. Given an interval domain $X\in\mathbb{IR}^n$ with
$\tilde{\xx}\in X$, and an interval matrix
$M\in\mathbb{IR}^{n\times n}$ satisfies $\nabla f_i(X)\subseteq
M_{i,:}$, for $i=1,\ldots,n$. Denote by $I$ the $n\times n$
identity matrix and assume
\begin{equation*}
-J_F(\tilde{\xx})F(\tilde{\xx})+(I-J_F(\tilde{\xx})M)X\subseteq
int(X).
\end{equation*}
Then there is a unique $\hat{\xx}\in X$ with $F(\hat{\xx})=0$.
Moreover, every matrix $\tilde{M}\in M$ is nonsingular. In
particular, the Jacobian matrix $J_F(\hat{\xx})$ is nonsingular.
\end{theorem}

In \citep{RuGr09}, they introduced a smoothing parameter to certify
a double root of a slightly perturbed system using Theorem
\ref{verification}. It should be noticed  that  a double root is the
simplest breath-one root with the multiplicity $2$.

In \citep{MM:2011}, they applied Theorem \ref{verification} to a
deflated system to verify a multiple root of a nearby system with
the computed local dual structure. Their method can deal with
arbitrary multiple roots.

For the breath one case and $\mu>2$, in \citep[Theorem
4.2]{RuGr09}, they proved that  it is impossible to compute an
inclusion of a multiple root by adding only a smoothing parameter
to one selected equation. We show below how to construct a
deflated  system  using the parameterized basis in
 $\mathbf{a}_2, \ldots, \mathbf{a}_{\mu}$ for the local dual space of
$\mathcal{D}_{\hat{\xx}}$ to certify breath-one multiple roots for
$\mu \geq 2$.
%
%
% as we
%notice from the former section that,  once the variable to apply
%the normalization and reduction is decided, the form of a basis
%$\{\Lambda_1,\Lambda_2,\ldots,\Lambda_{\mux}\}$ of
%$\mathcal{D}_{\hat{\xx}}$ is fixed, if we regard $\mathbf{a}_k$,
%$k=2,\ldots,\mux$ as unknowns. In other words, it means that we
%have a parameterized basis, besides, we can use this basis to
%construct deflated systems and certify multiple roots.

Let $F=\{f_1,\ldots,f_n\}\in R$ be given.
Suppose $\hat{\xx}\in \mathbb{K}^n$ is an isolated root of $F$ with the
multiplicity $\mux$ and $J_F(\hat \xx)$ has
corank one.  We show first  how to choose  a pair of suitable variable and
equation to perform the perturbation. In fact, as we showed
before, the perturbed variable $x_i$ can be determined by choosing
a column of $J_F(\hat{\xx})$, which can be written as a linear
combination of the other $n-1$ columns.  Similarly, suppose the
$j$-th row of $J_F(\hat{\xx})$ can be written as a linear
combination of the other $n-1$ linearly independent rows,
then we add  the perturbed univariate polynomial in $x_i$ to $f_j$.
 Then we perform
\begin{equation}\label{exchangexifj}
x_1\leftrightarrow x_i\mbox{ and }f_1\leftrightarrow f_j
\end{equation}
 to
construct the deflated system in (\ref{deflation}).

\begin{assumption}\label{choosexifj}
Suppose $J_F(\hat \xx)$ has corank one.  We  assume below that  the
first row (column) of $J_F(\hat \xx)$ can be written as a linear
combination of its other rows (columns). This can always be achieved
by changing of variables and renumbering equations as above.
\end{assumption}

%
%
%We assume the system has been normalized and reduced to a system
%satisfies the assumption (\ref{initialization}).
%
%
%We suppose $x_1$ is the variable to apply the normalization and
%reduction, then

We introduce $\mux-1$ smoothing parameters
$b_0,b_1,\ldots,b_{\mux-2}$ and construct a deflated system
$G(\mathbf{x},\mathbf{b},\mathbf{a})$ with  $\mu n$  variables and
$\mu n$
 equations:
\begin{equation}\label{deflation}
G(\mathbf{x},\mathbf{b},\mathbf{a})=\left(
\begin{array}{c}
  F_1(\mathbf{x},\mathbf{b})=F(\mathbf{x})-\left(\sum_{\nu=0}^{\mu-2} \frac{b_{\nu}x_1^{\nu}}{\nu!}\right)\mathbf{e}_1   \\
  F_2(\mathbf{x},\mathbf{b},\mathbf{a}_2)   \\
  F_3(\mathbf{x},\mathbf{b},\mathbf{a}_2,\mathbf{a}_3)  \\
  \vdots  \\
  F_{\mux}(\mathbf{x},\mathbf{b},\mathbf{a}_2,\ldots,\mathbf{a}_{\mux})
\end{array}
\right),
\end{equation}
where $\mathbf{b}=[b_0,b_1,\ldots,b_{\mu-2}]$,
$\mathbf{a}=[\mathbf{a}_2,\mathbf{a}_3,\ldots,\mathbf{a}_{\mux}]$,
$\mathbf{a}_2=[1,a_{2,2},\ldots,a_{2,n}]^T$,
$\mathbf{a}_k=[0,a_{k,2},\ldots,a_{k,n}]^T$ for $2 <k \leq \mu$, and
\begin{equation}\label{Fk}
F_{k}(\mathbf{x},\mathbf{b},\mathbf{a}_2,\ldots,\mathbf{a}_{k})=L_{k}(F_1).
\end{equation}

%Suppose $G(\hat{\xx},\hat{\bb},\hat{\aaa})=0$,  %is a root of
%%$G(\mathbf{x},\mathbf{b},\mathbf{a})=0$,
% we have the following
%theorem.

\begin{theorem}\label{nonsingular}
Suppose $G(\hat{\xx},\hat{\bb},\hat{\aaa})=0$. Under Assumption
\ref{choosexifj}, if the Jacobian matrix
$J_G(\hat{\xx},\hat{\bb},\hat{\aaa})$ is nonsingular, then
$\hat{\xx}$ is an isolated root of the polynomial
$F_0(\mathbf{x})=F_1(\mathbf{x},\hat{\bb})$ with the multiplicity
$\mux$ and the corank of $J_{F_0}(\hat \xx)$ is one.

\end{theorem}
\begin{proof}
From $G(\hat{\xx},\hat{\bb},\hat{\aaa})=0$,  %is a root of
%$G(\mathbf{x},\mathbf{b},\mathbf{a})=0$,
we have $F_0(\hat{\xx})=0$ and
\begin{equation*}
F_2(\hat{\xx},\hat{\bb},\hat{\aaa}_2)=J_{F_0}(\hat{\xx})\cdot\hat{\aaa}_2=0.
\end{equation*}
Since $\hat{\aaa}_2\neq 0$, we derive that
\begin{equation*}
\rank(J_{F_0}(\hat{\xx}))\leq n-1.
\end{equation*}
Moreover, from the expression of $\hat{\aaa}_2$, we know that the
first column of $J_{F_0}(\hat{\xx})$ can be written as a linear
combination of the other $n-1$ columns. Therefore,
\[\rank(\SJ(\hat{\xx}))=\rank(J_{F_0}(\hat{\xx}))\leq n-1,\]
where $\SJ(\hat{\xx})$ consists of the last $n-1$ columns of
$J_{F_0}(\hat{\xx})$. Similarly, since
$F_{k}(\hat{\xx},\hat{\bb},\hat{\aaa}_2,\ldots,\hat{\aaa}_{k})=0,$
by Theorem \ref{withouttran} and (\ref{newL}), we derive that
\begin{equation*}
\rank(\Delta_k(F_0),\SJ(\hat{\xx}))=\rank(\SJ(\hat{\xx}))\leq
n-1,\mbox{ for } 2 <k \leq \mu.
\end{equation*}
In order to prove that $\hat{\xx}$ is a breadth-one root of
$F_0(\mathbf{x})=0$ with the multiplicity $\mux$, we need to show
that
\begin{equation}\label{criterionb1}
\rank(\SJ(\hat{\xx}))=n-1\mbox{ and
}\rank(\Delta_{\mu+1}(F_0),\SJ(\hat{\xx}))=n.
\end{equation}

 It is interesting to see  that we can use the equivalent relations
\begin{equation}\label{equivaij}
\frac{\partial F_k}{\partial a_{i,j}}=\frac{\partial
F_{k-i+1}}{\partial x_j},\mbox{ for }1\leq i\leq k\mbox{ and
}1\leq j\leq n,
\end{equation}
to obtain  a simplified  expression of $J_G$%. %We can prove which can be
%The Jacobian matrix of $G$ is
\begin{equation}\label{JG}
J_G=\left(
\begin{array}{cccccccccc}
J_{F_1} & \ee_1 & x_1\ee_1 & \cdots & \frac{x_1^{\mux-2}}{(\mux-2)!}\ee_1 & 0 & 0 & \cdots & 0 & 0 \\
J_{F_2} & 0 & \ee_1 & \cdots & \frac{x_1^{\mux-3}}{(\mux-3)!}\ee_1 & \SJ & 0 & \cdots & 0 & 0 \\
J_{F_3} & 0 & 0 & \cdots & \frac{x_1^{\mux-4}}{(\mux-4)!}\ee_1 & {\tilde J}_{F_2} & \SJ & \cdots & 0 & 0 \\
\vdots & \vdots & \vdots & \ddots & \vdots & \vdots & \vdots & \ddots & \vdots & \vdots\\
J_{F_{\mux-1}} & 0 & 0 & \cdots & \ee_1 & {\tilde J}_{F_{\mux-2}} & {\tilde J}_{F_{\mux-3}} & \cdots & \SJ & 0 \\
J_{F_{\mux}} & 0 & 0 & \cdots & 0 & {\tilde J}_{F_{\mux-1}} & {\tilde J}_{F_{\mux-2}} & \cdots & {\tilde J}_{F_2} & \SJ \\
\end{array}
\right),
\end{equation}
where $J_{F_k}$ denotes the Jacobian matrix of
$F_k(\xx,\bb,\aaa_2,\ldots,\aaa_k)$ with respect to $\xx$ and
${\tilde J}_{F_k}$ consists of the last $n-1$ columns of
$J_{F_k}$, for $2\leq k\leq\mux$.
%easily proved by induction.

 If $\rank(\SJ(\hat{\xx}))\leq n-2$,
there will exist a nontrivial vector in its kernel. Note that
$\SJ(\hat{\xx})$ is the only non-zero element in the last column,
we claim that there will exist a nontrivial vector in the kernel
of $J_G(\hat{\xx},\hat{\bb},\hat{\aaa})$, which is a
contradiction. Then we derive that
\begin{equation*}
\rank(J_{F_0}(\hat{\xx}))=\rank(\Delta_k(F_0),\SJ(\hat{\xx}))=\rank(\SJ(\hat{\xx}))=n-1.
\end{equation*}
On other hand, from (\ref{newconstructP}) and (\ref{Fk}), we
derive that
\[J_G(\hat{\xx},\hat{\bb},\hat{\aaa})_{:,1..(\mux-1)n+1}\cdot\vv=(0,\ldots,0,\Delta_{\mu+1}(F_0))^T,\]
where
\[\vv=\frac{1}{\mux}\cdot(1,\hat{a}_{2,2},\ldots,\hat{a}_{2,n},0,\ldots,0,2\hat{a}_{3,2},\ldots,2\hat{a}_{3,n},\ldots,(\mux-1)\hat{a}_{\mux,2},\ldots,(\mux-1)\hat{a}_{\mux,n})^T.\]
So that, if $\rank(\Delta_{\mux+1}(F_0),\SJ(\hat{\xx}))\leq n-1$,
there will exist a nontrivial vector in the kernel of
$J_G(\hat{\xx},\hat{\bb},\hat{\aaa})$, which is a contradiction.
Hence,  we have
\[\rank(\Delta_{\mux+1}(F_0),\SJ(\hat{\xx}))=n.\]
Therefore, according to Theorem \ref{withouttran}, $\hat{\xx}$ is
an  isolated breath-one singular root of $F_0(\mathbf{x})=0$ with
the multiplicity $\mux$.

\end{proof}

%
% This is mainly caused by the multiplicity is
%greater than two. For our method, if the Jacobian matrix
%$J_F(\hat{\xx})$ is of corank one, we can always find the variable
%and the equation to perform the perturbation, which ensure that a
%inclusion of $\hat{\xx}$ can be computed.

\begin{theorem}\label{otherdirection}
 Suppose $\hat{\xx}$ is an exact
isolated root of $F(\mathbf{x})=0$ with the multiplicity $\mux$ and
the corank of $J_F(\hat \xx)$ is one exactly. Under Assumption
\ref{choosexifj}, we have
\begin{equation}\label{chooseeq}
\rank(\SJF(\hat{\xx}),\ee_1)=n.
\end{equation}
Evaluating (\ref{deflation}) at  $\hat{\bb}=\mathbf{0}$, i.e., no
perturbations for $F$,  the Jacobian matrix
$J_G(\hat{\xx},\mathbf{0},\hat{\aaa})$  is
nonsingular. % where $\hat{\bb}=\mathbf{0}$.
\end{theorem}
\begin{proof}
According to  Assumption  \ref{choosexifj}, the first row of
 $\SJF(\hat{\xx})$ can be written as a linear combination of its other rows. Since the rank of
 $\SJF(\hat{\xx})$ is $n-1$, its last $n-1$ rows must be linear
 independent.  Therefore, we have (\ref{chooseeq}).

 Assume $\vv$ is a nontrivial vector in the kernel of
$J_G(\hat{\xx},\mathbf{0},\hat{\aaa})$. If $v_1=0$, by checking the
columns of $J_G$ in (\ref{JG}), and using
(\ref{chooseeq}), %$\rank(\SJF(\hat{\xx}),\ee_1)=n$,
 we can show  that $\vv=0$. If $v_1\neq 0$, we can assume $v_1=1$.
Similar to the second part of proof of Theorem \ref{nonsingular}, we
derive that $\Delta_{\mu+1}(F)$ can be written as a linear
combination of the columns from $\SJF(\hat{\xx})$, which is a
contradiction. Hence, there exists no nontrivial vector in the
kernel of $J_G(\hat{\xx},\mathbf{0},\hat{\aaa})$. In other word,
$J_G(\hat{\xx},\mathbf{0},\hat{\aaa})$ is nonsingular.
\end{proof}

Now, we apply Theorem \ref{verification} on the deflated system. If the test succeeds,
we will derive verified and narrowed error bounds with the property that a slightly perturbed system
 is proved to have a breadth-one multiple root within the computed bounds.

\begin{theorem}\label{verifybreadth1}
Suppose Theorem \ref{verification} is applicable to
$G(\mathbf{x},\mathbf{b},\mathbf{a})$ in (\ref{deflation}) and
yields inclusions for $\hat{\xx}$, $\hat{\bb}$ and $\hat{\aaa}$
such that $G(\hat{\xx},\hat{\bb},\hat{\aaa})=0$. Then $\hat{\xx}$
is an isolated breath-one root of
$F_0(\mathbf{x}):=F_1(\mathbf{x},\hat{\bb})$ with the multiplicity
$\mux$.
\end{theorem}
\begin{proof}
A direct result of Theorem  \ref{verification} and Theorem
\ref{nonsingular}.
\end{proof}

%The following example is cited from  \citep{RuGr09}, in which no
%inclusion of its root $(0,0)$ can be summarized by their method,
%since $(0,0)$ is a breath-one singular root with  the multiplicity
%$4$ not $2$.

\begin{example}\citep[Example 4.11]{RuGr09}
Consider a polynomial system
\[F=\{x_1^2x_2-x_1x_2^2,x_1-x_2^2.\}\]
The system $F$ has $(0,0)$ as a $4$-fold isolated zero.
\end{example}
The Jacobian matrix of $F$ at $(0,0)$ is
\[\JF(0,0)=\left[\begin{array}{cc}
  0 & 0  \\
  1 & 0
\end{array}
\right],\] so we choose $x_2$ as the perturbed variable and  add the
 univariate polynomial $b_0+b_1x_2+\frac{b_2}{2}x_2^2$
  to the first equation in $F$ to construct the
parameterized deflated system
\[G(\mathbf{x},\mathbf{b},\mathbf{a})=\left(\begin{array}{c}
            x_1^2x_2-x_1x_2^2-b_0-b_1x_2-\frac{b_2}{2}x_2^2 \\
            x_1-x_2^2 \\
            2a_1x_1x_2-a_1x_2^2+x_1^2-2x_1x_2-b_1-b_2x_2 \\
            a_1-2x_2 \\
            a_1^2x_2+2a_1x_1-2a_1x_2+2a_2x_1x_2-a_2x_2^2-x_1-\frac{b_2}{2} \\
            a_2-1 \\
            a_1^2+2a_1a_2x_2-a_1+2a_2x_1-2a_2x_2+2a_3x_1x_2-a_3x_2^2 \\
            a_3
          \end{array}
\right).\] Applying the INTLAB function \textsf{verifynlss}
\citep{RumpINT} to $G$ with the initial approximation
\[[0.002,0.003,-0.001,0.0015,-0.002,0.002,1.001,-0.01]\]
to obtain inclusions
  \begin{align*}
  &[\,\,\,\,-0.00000000000001,\,\,\,\,\,\,0.00000000000001]\\
  &[\,\,\,\,-0.00000000000001,\,\,\,\,\,\,0.00000000000001]\\
  &[\,\,\,\,-0.00000000000001,\,\,\,\,\,\,0.00000000000001]\\
  &[\,\,\,\,-0.00000000000001,\,\,\,\,\,\,0.00000000000001]\\
  &[\,\,\,\,-0.00000000000001,\,\,\,\,\,\,0.00000000000001]
%  &[\,\,\,\,-0.00000000000001,\,\,\,\,\,\,0.00000000000001]\\
%  &[\,\,\,\,\,\,\,\,\,1.00000000000000,\,\,\,\,\,\,1.00000000000000]\\
%  &[\,\,\,\,-0.00000000000001,\,\,\,\,\,\,0.00000000000001]
  \end{align*}
This proves that the perturbed system $F_0(\mathbf{x})$ ($|b_i| \leq 10^{-14}, i=0,1,2$) has a $4$-fold root $\hat{\xx}$ with $-10^{-14} \leq \hat{\xx} \leq 10^{-14}$.

\begin{example}\citep[Example 4.7]{RuGr09}\label{sensitive}
Consider a polynomial system
\[F=\{x_1^2-x_2^2,x_1-x_2^2.\}\]
The system $F$ has $(0,0)$ as a $2$-fold isolated zero.
\end{example}
For this example, as mentioned in \citep{RuGr09}, the iteration is
sensitive to the initial approximations. %since they apply the
%classic Newton method to the deflated system with the
%perturbation.
 Applying the INTLAB function \textsf{verifynlss2} to
$F$ with the starting point $[0.002,0.001]$, we will obtain
inclusions
  \begin{align*}
  &[\,\,\,\,-0.00000000000001,\,\,\,\,\,\,0.00000000000001]\\
  &[\,\,\,\,-0.00000000000001,\,\,\,\,\,\,0.00000000000001]\\
  &[\,\,\,\,-0.00000000000001,\,\,\,\,\,\,0.00000000000001]
  \end{align*}
However, for the initial approximation $[0.001,0.001]$, we obtain
  \begin{align*}
  &[\,\,\,\,\,\,\,\,\,0.49999999999999,\,\,\,\,\,\,0.50000000000001]\\
  &[\,\,\,\,\,\,\,\,\,0.70710678118654,\,\,\,\,\,\,0.70710678118655]\\
  &[\,\,\,\,-0.25000000000001,\,\,-0.24999999999999]
  \end{align*}
which finds the double root $(0.5,1/\sqrt{2})$ of $x_1^2-x_2^2+0.25=0$ and $x_1-x_2^2=0$.

For this reason, we prefer to use the symbolic-numeric method described
 in \citep{LZ:2011}  to refine initial approximations firstly, then use the method in Theorem \ref{verifybreadth1} to compute inclusions of multiple roots.
We show the routine \verb"MRRB1" below for refining a singular
solution to  high precision.  %Compared with the algorithm in
%\citep{LZ:2011},
 The input of \verb"MRRB1" is a sequence of
polynomial systems $F_1,F_2,\ldots,F_{\mux}$  defined in
(\ref{deflation}) (\ref{Fk}) with $\bb=\mathbf{0}$,
$F_{\mux+1}=P_{\mux+1}(F_1)$ and an approximate solution of $F_1=0$.
The algorithm in \citep{LZ:2011} has been improved in \verb"MRRB1"
by avoiding linear transformations and constructing differential
functionals repeatedly.
%The differential functions are also computed more efficiently.

\begin{algorithm}\label{MRRB1}MRRB1

\noindent \textbf{Input:} A sequence of systems
$F_1,\ldots,F_{\mux+1}$, a point $\hat{\xx}\in\mathbb{K}^n$.

\noindent \textbf{Output:} A refined point $\hat{\xx}$ and refined
parameters $\hat{\aaa}_2,\ldots,\hat{\aaa}_{\mux}$.

\begin{enumerate}

\item \label{s1} \verb"Regularized Newton Iteration": Solve the
least squares problem
\[\left(J_{F_1}^{\tr}(\hat{\xx})\cdot
J_{F_1}(\hat{\xx})+\sigma_nI_n\right)
\yy=-J_{F_1}^{\tr}(\hat{\xx})\cdot F_1(\hat{\xx}),\]
%
%\[\hat{\xx}:=\hat{\xx}+\emph{LeastSquares}\left(\left(J_{F_1}^{\tr}(\hat{\xx})\cdot
%J_{F_1}(\hat{\xx})+\sigma_nI_n\right)
%\xx=J_{F_1}^{\tr}(\hat{\xx})\cdot F_1(\hat{\xx})\right),\]
 where
$J_{F_1}^{\tr}(\hat{\xx})$ is the conjugate transpose of
$J_{F_1}(\hat{\xx})$, $\sigma_n$ is the smallest singular value of
$J_{F_1}(\hat{\xx})$ and $I_n$ is the $n\times n$ identity matrix.

Set $\hat{\xx}:=\hat{\xx}+ \hat{\yy}$.

\item \label{s2} For $2 \leq k\leq\mux$, solve the least squares
problem
\[F_k(\hat{\xx},\hat{\aaa}_2,\ldots,\hat{\aaa}_{k-1},\aaa_k)=0\]
 to obtain $\hat{\aaa}_k$.
%\[\hat{\aaa}_k:=\emph{LeastSquares}\left(F_k(\hat{\xx},\hat{\aaa}_2,\ldots,\hat{\aaa}_{k-1},\aaa_k)=0\right)\mbox{and }k:=k+1.\]

\item \label{s3} Solve the linear system
\[\left[F_{\mux+1}(\hat{\xx},\hat{\aaa}_2,\ldots,\hat{\aaa}_{\mux}),\frac{\partial F_1(\hat{\xx})}{\partial x_2},\ldots,\frac{\partial F_1(\hat{\xx})}{\partial x_n}\right]\vv=-F_{\mux}(\hat{\xx},\hat{\aaa}_2,\ldots,\hat{\aaa}_{\mux}),\]
where $\vv=(v_1,\ldots,v_n)^T$. Set $\delta:=v_1/\mux$.

\item Return $\hat{\aaa}_2,\ldots,\hat{\aaa}_{\mux}$ and
\[\hat{\xx}:=\hat{\xx}+\delta\left(\begin{array}{c}
                                                 1 \\
                                                 \hat{a}_{2,2} \\
                                                 \vdots \\
                                                 \hat{a}_{2,n}
                                               \end{array}
\right).\]

%\[\hat{\xx}:=\hat{\xx}+\delta\left[\begin{array}{c}
%                                                 1,
%                                                 \hat{a}_{2,2},
%                                                 \ldots,
%                                                 \hat{a}_{2,n}
%                                               \end{array}
%\right]^T.\]

\end{enumerate}

\end{algorithm}
Now we consider Example \ref{sensitive} again. For $[0.002,0.001]$, after running \verb"MRRB1" two times then applying the INTLAB function \textsf{verifynlss} to $G$, we obtain
\begin{align*}
  &[\,\,\,\,-0.00000000000001,\,\,\,\,\,\,0.00000000000001]\\
  &[\,\,\,\,-0.00000000000001,\,\,\,\,\,\,0.00000000000001]\\
  &[\,\,\,\,-0.00000000000001,\,\,\,\,\,\,0.00000000000001]
  \end{align*}
Similarly, for $[0.001,0.001]$, we obtain
\begin{align*}
  &[\,\,\,\,-0.00000000000001,\,\,\,\,\,\,0.00000000000001]\\
  &[\,\,\,\,-0.00000000000001,\,\,\,\,\,\,0.00000000000001]\\
  &[\,\,\,\,-0.00000000000001,\,\,\,\,\,\,0.00000000000001]
  \end{align*}

\begin{example}\label{ref2} \citep{LiZhi:2009}
Consider a   system $F=\{f_1, \ldots, f_s\}$ given by
\vspace{-0.5cm}
\begin{eqnarray*}
f_i &=& x_i^2+x_i-x_{i+1}, ~{\mbox{if}}~ i < s,\\
f_s &=& x_s^3
 \end{eqnarray*}
 with a breath-one singular zero $(0, 0, \ldots,0)$  of  the multiplicity $3$. %The polynomial system is sparse and the time for computing the differential conditions increases cubic in $s$.
\end{example}

We run \verb"MRRB1" three times for initial approximate roots near
the origin, whose errors are around $10^{-4}$, to obtain refined
$\hat{\xx}$ and $\hat{\aaa}$ with errors about $10^{-12}$. We
choose $x_s$ as the perturbed variable and add the univariate
polynomial $b_0+b_1 x_s$ to the last polynomial  $f_s$ to
construct the parameterized deflated system. In the following
table, $|X|$ and $|B|$ denotes the interval size of inclusions for
$\hat{\xx}$ and $\hat{\bb}$ computed by applying INTLAB function
\textsf{verifynlss} to the deflated system (\ref{deflation}) and
$(\hat{\xx},\mathbf{0},\hat{\aaa})$.
\begin{table}[ht]
\begin{center}
\begin{tabular}{|c|c|c|} \hline
s & $|X|$ & $|B|$ \\ \hline
10 & $10^{-14}$ & $10^{-14}$ \\ \hline
20 & $10^{-14}$ & $10^{-14}$ \\ \hline
50 & $10^{-14}$ & $10^{-14}$ \\ \hline
100 & $10^{-14}$ & $10^{-14}$ \\ \hline
200 & $10^{-12}$ & $10^{-12}$ \\ \hline
500 & $10^{-12}$ & $10^{-12}$ \\ \hline
1000 & $10^{-12}$ & $10^{-12}$ \\ \hline
\end{tabular}
%{\label{table1}Algorithm Performance}
\end{center}
\vspace{-0.4cm}
\end{table}

%\section{Numerical Results}
%\label{num}

%\section{Conclusion}
%\label{con}

%\section*{Acknowledgments}
%Lihong Zhi would like to thank  Gilles Villard and  Nathalie Revol
%  for hosing her
%visit to LIP-ENS de Lyon in 2010. The paper is stimulated by
%fruitful discussions during the visit.

%\bibliographystyle{elsart-harv}
\bibliographystyle{elsarticle-harv}
\bibliography{strings,wuxiaoli,linan,zhi}
\end{document}